\documentclass[12pt, reqno, a4paper]{amsart}
\usepackage[margin=1.2in]{geometry}
\numberwithin{equation}{section}
\usepackage{amssymb,amsfonts,amsthm}
\usepackage[utf8]{inputenc}
\usepackage{listings}
\usepackage{bm}
\usepackage[hyperpageref]{backref}
\usepackage{esint}
\usepackage{color}
\usepackage[bookmarks]{hyperref}
\usepackage{siunitx}
\usepackage{bigints}
\usepackage{hyperref}
\usepackage{mathtools}
\usepackage{tikz}

\allowdisplaybreaks[4]

\newtheoremstyle{myremark}{10pt}{10pt}{}{}{\bfseries}{.}{.5em}{}

 \newtheorem{thm}{Theorem}
 \newtheorem{cor}{Corollary}[section]
 \newtheorem{lemma}{Lemma}[section]
 \newtheorem{prop}[lemma]{Proposition}
 \theoremstyle{definition}
 \newtheorem{defn}{Definition}

\usepackage{hyperref}

 \newcommand{\norm}[1]{\left\Vert#1\right\Vert}

\allowdisplaybreaks[4]

\begin{document}

\title[Trudinger-Moser type inequality]{Trudinger-Moser type inequality in fractional Sobolev space with singularity on smooth submanifold}

\author{VIVEK SAHU}

\address{ Theoretical Statistics and Mathematics Unit,
Indian Statistical Institute, Delhi Centre, S.J. Sansanwal Marg, New Delhi, Delhi 110016, India}

\email{vivek@isid.ac.in, viiveksahu@gmail.com}



\date{}

\dedicatory{}

\begin{abstract}
We prove a Trudinger--Moser type inequality in fractional Sobolev spaces with singularities on smooth compact sets of codimension $k$, where $1 < k < d$ and $sp = d$. The singular term is given by the inverse $d$-th power of the distance to the submanifold. The proof is based on a fractional Hardy inequality adapted to smooth submanifolds, and we show the sharpness of the constant. We also establish the equivalence of two natural fractional Sobolev spaces vanishing on the singular set.
\end{abstract}

\maketitle


\section{Introduction}
The Trudinger inequality, introduced in $1967$ by Trudinger \cite{trudinger1967}, is a well-known result in the study of Sobolev spaces. It shows that functions in $W^{1,d}(\Omega)$ can grow exponentially and still be integrable. This extends the usual Sobolev embedding results to Orlicz spaces. Later, in $1971$, J. Moser \cite{moser1970} improved this result by giving a sharper version, now known as the Trudinger-Moser inequality. This result also provides the best constant. It says that for a bounded domain $\Omega \subset \mathbb{R}^{d}$,
\begin{equation*}
    \sup_{ u \in W^{1,d}_{0}(\Omega), \,  \| \nabla u \|_{L^{d}(\Omega)} \leq 1 }  \int_{\Omega} \exp \left( \alpha |u|^{\frac{d}{d-1}} \right) \, dx < \infty, \hspace{3mm} \forall \ 0 \leq \alpha \leq \alpha_{d},
\end{equation*}
where $\alpha_{d} = d (d \omega_{d})^{\frac{1}{d-1}}$ and $\omega_d$ denotes the volume of the unit ball in $\mathbb{R}^{d}$. If $\alpha> \alpha_{d}$, the integral become infinite, making $\alpha_{d}$ the best possible constant. For further development in Trudinger-Moser inequality, we refer to \cite{Adachi2000, Adi2007, Adi2010, Ruf2015, pro2021, Nguyen2013}. 

\smallskip

In this article, we aim to extend our previous result \cite{Adi2025} and establish Trudinger-Moser type inequality in fractional Sobolev space with singularity on smooth compact set $K$ of codimension $k$, where $k \in \mathbb{N}$ and $1<k<d$. In 2019, E. Parini and B. Ruf \cite{ruf2019} extended the Trudinger-Moser inequality to fractional Sobolev spaces in bounded Lipschitz domains for $d \geq 2$. They proved that for a bounded Lipschitz domain $\Omega$ in $\mathbb{R}^{d}$ with $d \geq 2$, and $sp = d$, there exists a constant $\alpha_* = \alpha_*(d,s, \Omega) > 0$ such that for all $\alpha \in [0, \alpha_*)$,
\begin{equation}
    \sup_{u \in \widetilde{W}^{s,p}_0(\Omega), \, [u]_{W^{s,p}(\mathbb{R}^d)} \leq 1} \int_\Omega \exp\left(\alpha |u(x)|^{\frac{d}{d-s}}\right) \, dx < \infty.
\end{equation}
Moreover, the above inequality fails for any $\alpha \in (\alpha^{*}_{s,d}, \infty)$, where
\begin{equation}\label{Defn alpha star d}
\alpha^{*}_{s,d} :=    d  \left(  \frac{2 (d \omega_{d})^{2} \Gamma (p+1) }{d!}  \sum_{n=0}^{\infty} \frac{(d+n-1)!}{n!} \frac{1}{(d+2n)^{p}} \right)^{\frac{s}{d-s}}.
\end{equation}
Here, $\widetilde{W}^{s,p}_0(\Omega)$ is the completion of $C^\infty_c(\Omega)$ under the norm $(\|\cdot\|_{L^{p}(\Omega)} + [\cdot]_{W^{s,p}(\mathbb{R}^d)})^{\frac{1}{p}}$, where $[\cdot]_{W^{s,p}(\mathbb{R}^{d})}$ is Gagliardo seminorm on $\mathbb{R}^{d}$. Determining the optimality of $\alpha^{*}_{s,d}$ remains an interesting open problem in the literature. S. Iula, in \cite{Iula2017}, investigated the one-dimensional case ($d = 1$) and established the Trudinger-Moser inequality within the framework of fractional Sobolev spaces in dimension one.

\smallskip

N. V. Thin, in \cite{Thin2020}, extended the above result to the singular Trudinger–Moser inequality with the weight $\frac{1}{|x|^{\gamma}}$, where $0 < \gamma < d$ in $\mathbb{R}^{d}$. He proved that for $0 \leq \gamma < d$ and $sp = d$, there exists $\beta_{*} > 0$ such that for any $0 \leq \alpha \leq \beta_{*} < \alpha_{*}$, the following inequality holds:
\begin{equation*}
    \sup_{u \in W^{s,p}(\mathbb{R}^{d}), ~ \| u \|_{W^{s,p}(\mathbb{R}^{d})} \leq 1} \int_{\mathbb{R}^{d}} \frac{\Phi_{d,s} \left( \alpha |u(x)|^{\frac{d}{d-s}} \right) }{|x|^{\gamma}} \, dx <  \infty, 
\end{equation*}
where
\begin{equation}\label{Defn: Phi}
    \Phi_{d,s} (t):= \exp(t) - \sum_{i=0}^{j_{p}-2} \frac{t^{i}}{i !}, \hspace{0.5cm} j_{p} = \min \left\{ j \in \mathbb{N} : j \geq p \right\}.
\end{equation}

The author, together with Adimurthi and P. Roy in \cite{Adi2025}, established a weighted Trudinger type inequality in the context of fractional Hardy inequality with boundary singularities. In \cite{Adi2025}, it is proved that for a bounded Lipschitz domain $\Omega \subset \mathbb{R}^{d}$, where $d \geq 2$ with $sp = d$, there exists $\alpha_{*} > 0$ such that for any $0 \leq \alpha < \alpha_{*}$, the following inequality holds:
\begin{equation*}
    \sup_{u \in W^{s,p}_{0}(\Omega), ~ [u]_{W^{s,p}(\Omega)} \leq 1} \int_{\Omega} \Phi_{d,s} \left( \alpha |u(x)|^{\frac{d}{d-s}} \right) \frac{dx}{\delta^{d}_{\Omega}(x)} \, dx  <  \infty, 
\end{equation*}
where $\delta_{\Omega}(x) := \min_{y \in \partial \Omega} |x - y|$ and $\Phi_{d,s}$ is defined in \eqref{Defn: Phi}. Furthermore, for the case $d = 1$, the same article \cite{Adi2025} establishes an appropriate weighted Trudinger-type inequality with a suitable logarithmic weight function. In short, the following inequality holds for any $0 \leq \alpha < \alpha_{*}$:
\begin{equation*}
  \sup_{u \in W^{s,p}(\Omega), ~ [u]_{W^{s,p}(\Omega)} \leq 1} \int_{\Omega} \Phi_{d,s} \left( \alpha \left( \frac{|u(x)-(u)_{\Omega}|}{\ln \left( \frac{2R}{\delta_{\Omega}(x)} \right) } \right)^{\frac{1}{1-s}} \right) \frac{dx}{\delta_{\Omega}(x)} \, dx  <  \infty,   
\end{equation*}
where $(u)_{\Omega}$ is the average of $u$ over $\Omega$. Furthermore, it was established that for any $\alpha> \alpha^{*}_{s,d}$, the above two inequalities fails. The critical case $sp=d$ with $d=1$ is motivated by the work of Adimurthi, P. Jana, and P. Roy \cite{AdiPurPro2026}, where the authors established a fractional Hardy inequality with an optimal logarithmic weight function. Using this, the above Trudinger-type inequality was established for the one-dimensional case.

\smallskip

This article extends the above Trudinger-type inequality in the context of fractional Hardy inequality with a boundary singularity (see \cite{Adi2025}) to the case where the singularity lies on a smooth submanifold of codimension $k$, with $k \in \mathbb{N}$ and $1 < k < d$. In \cite{adimurthiSubmanifold}, the author, along with Adimurthi and P. Roy, proved fractional Hardy inequalities with singularities on smooth submanifolds. In more recent work with Kijaczko \cite{kijaczko2025} (see also \cite{Vivek2025}), we proved several fractional Hardy inequalities with singularities on flat submanifolds  with an optimal constant. Using the definition of a smooth submanifold from \cite{adimurthiSubmanifold}, we now establish Trudinger-Moser type inequality with singularities on such submanifolds. Let $\Omega$ be a bounded open set in $\mathbb{R}^{d}$, and let $K \subset \Omega$ be a compact set. For any $x \in \Omega$, we define
\begin{equation}
\delta_{K}(x) := \inf_{y \in K} |x-y|.
\end{equation}

We adopt the following definition of a smooth compact set of codimension $k$, where $k \in \mathbb{N}$ and $1 < k < d$:

\begin{defn}\label{definition}
    For any $k \in \mathbb{N}$ and $1<k< d$, let $K$ be a compact surface of codimension $k$ in $\mathbb{R}^{d}$,  we say that $K$  is a class of $C^{0,1}$ if there exists $C>0$ such that for any $x \in K$ there exists a ball $B=B_{r}(x) \subset \mathbb{R}^{d}, ~ r >0$ , and a isomorphism $T: Q \to B$ such that
    \begin{equation*}
        T \in C^{0,1}(\overline{Q}), \hspace{3mm} T^{-1} \in C^{0,1}(\overline{B}) \hspace{3mm} \text{and} \hspace{3mm} \norm{T}_{C^{0,1}(\overline{Q})} + \norm{T^{-1}}_{C^{0,1}(\overline{B})} \leq C,
    \end{equation*}
    where  $Q= B^{k}_{1}(0) \times (0,1)^{d-k}, ~ B^{k}_{1}(0) \subset \mathbb{R}^{k}$ is a unit ball with center  $0$ (see Figure \ref{fig:myfigure1}). For any  $\xi \in Q= B^{k}_{1}(0) \times (0,1)^{d-k}$, we denote as  $\xi= (\xi_{k}, \xi_{d-k})$, where  $\xi_{k} \in B^{k}_{1}(0)$ and  $\xi_{d-k} \in (0,1)^{d-k}$. Also,  $T^{-1}(B \cap K) =  \{ (\xi_{k}, \xi_{d-k}) \in Q : \xi_{k} =0 \}$. We also assume that for any  $y \in (\Omega \setminus K) \cap B$, we have
    \begin{equation*}
        \delta_{K}(y) \sim |\xi_{k}|,
    \end{equation*}
   i.e., $C_{1} |\xi_{k}| \leq \delta_{K}(y) \leq C_{2} |\xi_{k}|$ for some $C_{1}, ~ C_{2}>0$ very close to $1$, where $T((\xi_{k}, \xi_{d-k})) = y$. \\
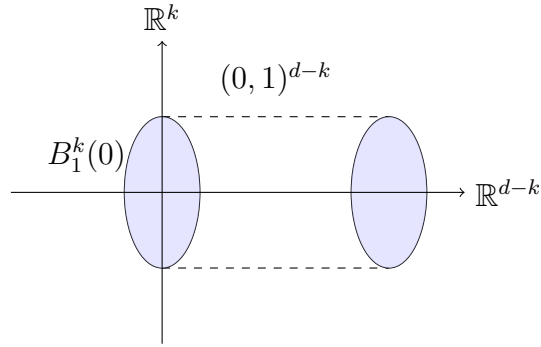
\begin{figure}[h!]
\begin{center}
\begin{tikzpicture}
\node at (-1,.5) {$B^{k}_{1}(0)$};
\node at (1.5,1.5) {$(0,1)^{d-k}$};

\draw[dashed] (0,1) -- (3,1);
\draw[dashed] (0,-1) -- (3,-1);

\draw (0,0) ellipse [y radius=1, x radius=0.5];
\draw (3,0) ellipse [y radius=1, x radius=0.5];
\fill[blue!20, opacity=0.5] (0,0) ellipse [y radius=1, x radius=0.5];
\fill[blue!20, opacity=0.5] (3,0) ellipse [y radius=1, x radius=0.5];

\draw[->] (-2,0) -- (4,0) node[right] {$\mathbb{R}^{d-k}$};
\draw[->] (0,-2) -- (0,2) node[above] {$\mathbb{R}^{k}$};
\end{tikzpicture}
\end{center}
\caption{$Q= B^{k}_{1}(0) \times (0,1)^{d-k}$.}
\label{fig:myfigure1}
\end{figure}
\end{defn}

In the above definition, the assumption $\delta_{K}(y) \sim |\xi_k|$ is not really an assumption. In Section \ref{Geometry of Compact Lipschitz Submanifolds}, we will further explain smooth compact sets of codimension $k$, in a manner similar to the boundary of a bounded Lipschitz domain. In that section, we will establish that $\delta_{K}(y) \sim |\xi_k|$ locally, in terms of a Lipschitz function. Therefore, we mention it in the definition because this relationship is an important component in establishing the main result of this article.

\smallskip

Consider an open set $\Omega \subset \mathbb{R}^{d}$, and let $K$ be a compact subset of $\Omega$. Let $p>1$ and $s \in (0,1)$. We define the space $W^{s,p}_{0}(\Omega \setminus K)$ as the closure of $C^{\infty}_{c} (\Omega \setminus K)$ functions with respect to the norm $\norm{\cdot}_{W^{s,p}(\Omega)}$ (see \eqref{norm defn} for the definition of $\norm{\cdot}_{W^{s,p}(\Omega)}$). For every $u\in W^{s,p}_{0}(\Omega\setminus K)$, we denote by
$\widetilde{u}:\mathbb{R}^{d}\to\mathbb{R}$ its zero extension, defined by
\begin{equation}\label{Extension Function}
   \widetilde{u}(x)=
\begin{cases}
u(x), & x\in\Omega\setminus K,\\[1mm]
0, & x\in (\mathbb{R}^{d}\setminus\Omega) \cup K.
\end{cases} 
\end{equation}
We then define
\begin{equation}\label{Defn: Extension Space}
  \widetilde{W}^{s,p}_{0}(\Omega\setminus K)
:=
\left\{
\widetilde{u}:\;
u\in W^{s,p}_{0}(\Omega\setminus K)
\right\}.  
\end{equation}

\smallskip

The following theorem is the main result of this article. We establish a suitable singular Trudinger--Moser type inequality in the fractional Sobolev space, where the singular term in the integrand is given by $\frac{1}{\delta^{d}_{K}(x)}$, with $K$ being a smooth, compact set of lower dimension.

\begin{thm}\label{Th : Trud-Moser inequ}
Let $\Omega$ be a bounded Lipschitz domain in $\mathbb{R}^{d}$, and let $K \subset \Omega$ be a smooth compact set of class $C^{0,1}$ of codimension $k$, where $k \in \mathbb{N}$ and $1<k<d$. Let $p>1$ and $s \in (0,1)$ be such that $sp=d$. Then there exists $\alpha_{*}=\alpha_{*}(d,s,k, \Omega, K)>0$ such that
 \begin{align}
   \sup_{u \in \widetilde{W}^{s,p}_{0}(\Omega \setminus K), \ [u]_{W^{s,p}(\mathbb{R}^{d})} \leq 1}  \int_{\Omega} \Phi_{d,s} \left( \alpha  |u(x)|^{\frac{d}{d-s}}  \right) \, \frac{dx}{\delta^{d}_{K}(x)}     < \infty , \quad   \forall \ \alpha \in [0, \alpha_{*}),
\end{align}
where $\Phi_{d,s}$ is as defined in \eqref{Defn: Phi}. Furthermore, the above inequality fails for any $\alpha> \alpha^{*}_{s,d}$, where $\alpha^{*}_{s,d}$ is defined in \eqref{Defn alpha star d}.
\end{thm}

We establish in Corollary \ref{corollary 1} that, for any $\alpha > 0$,
\begin{equation}
\int_{\Omega} \Phi_{d,s} \left( \alpha |u(x)|^{\frac{d}{d-s}} \right) \, \frac{dx}{\delta^{d}_{K}(x)} < \infty.
\end{equation}
Furthermore, the spaces $W^{s,p}_{0}(\Omega \setminus K)$ and $\widetilde{W}^{s,p}_{0}(\Omega \setminus K)$ are equivalent with respect to the fractional Sobolev norm. This equivalence is proved in Proposition \ref{Proposition 1}.

\smallskip

The next theorem is a key result used to establish the existence of $\alpha_{*}$ in Theorem \ref{Th : Trud-Moser inequ}. It presents a fractional Hardy inequality with singularity on a smooth compact set of codimension $k$, where $k \in \mathbb{N}$ and $1 < k < d$. It is stated as follows:

\begin{thm}\label{Th: Fractional Hardy sp=k}
     Let $\Omega$ be a bounded Lipschitz domain in $\mathbb{R}^{d}$, and let $K \subset \Omega$ be a smooth compact set of class $C^{0,1}$ of codimension $k$, where $k \in \mathbb{N}$ and $1<k<d$. Let $p>1$ and $s \in (0,1)$ be such that $sp=d$. Then for any $\tau \geq p$, there exists a constant $C=C(d,s,k, \Omega, K)>0$ such that
    \begin{equation}
       \left( \bigintssss_{\Omega} \frac{|u(x)|^{\tau}}{\delta^{d}_{K}(x)} \, dx \right)^{\frac{1}{\tau}} \leq C \tau^{\frac{d-s}{d}} \left( \int_{\Omega} \int_{\Omega}  \frac{|u(x)-u(y)|^{p}}{|x-y|^{d+sp}} \, dx \, dy \right)^{\frac{1}{p}} , \hspace{3mm} \forall \ u \in W^{s,p}_{0}(\Omega \setminus K).
    \end{equation}
\end{thm}

The above theorem establishes a fractional Hardy inequality with singularity on a smooth compact set of codimension $ k$, where  $k \in \mathbb{N}$ and  $1 < k < d $. For further results on fractional Hardy inequalities with singularities on smooth submanifolds, we refer to our recent work with Adimurthi and P. Roy \cite{adimurthiSubmanifold}.  For recent developments in the theory of fractional Hardy inequalities, we refer to \cite{Adi2026, Bianchi2024, Brasco2018, dyda2004, Dyda2024-2, Dyda2024, frank2008} and the references therein.

\smallskip

This article is organised as follows. In Section~\ref{preliminaries}, we present a preliminary lemma that is already known in the literature and serves as a key tool in proving our main results. In Section~\ref{Geometry of Compact Lipschitz Submanifolds}, we study the geometry of smooth compact sets of codimension $k$ and describe their local structure. In Section~\ref{fractional hardy}, we establish a fractional Hardy inequality with singularity on a smooth submanifold in the critical case $sp = d$. Finally, in Section~\ref{Proof of Theorem 1}, we prove Theorem~\ref{Th : Trud-Moser inequ}, which provides a Trudinger--Moser type inequality in the fractional Sobolev space with singularity on a smooth compact set of lower dimension.



\section{Notation and Preliminaries}\label{preliminaries} 
In this section, we present the notation and preliminary lemmas that will be used to prove our main results.
Throughout this article, we use the following notation: for any bounded open set $\Omega$, 
\begin{equation*}
(u)_{\Omega} := \frac{1}{|\Omega|} \int_{\Omega} u(x) \, dx  = \fint_{\Omega} u(x) \, dx 
\end{equation*}
denotes the average value of $u$ over $\Omega$, where $|\Omega|$ is the Lebesgue measure of $\Omega$. For any $x \in \mathbb{R}^{d}$, we write $x=(x_{k}, x_{d-k})$, where $x_{k} \in \mathbb{R}^{k}$ and $x_{d-k} \in \mathbb{R}^{d-k}$.  One has, for any  $a_{1},  \dots , a_{m} \in \mathbb{R}$ and  $\gamma \geq1$,
\begin{equation}\label{sumineq}
    \sum_{\ell=1}^{m} |a_{\ell}|^{\gamma} \leq  \left( \sum_{\ell=1}^{m} |a_{\ell}| \right)^{\gamma} .
\end{equation}

\smallskip

For $p \in [1,\infty)$ and $s \in (0,1)$, the fractional Sobolev space $W^{s,p}(\Omega)$ is defined as
\begin{equation*}
W^{s,p}(\Omega) := \left\{    u \in L^{p}(\Omega) : \int_{\Omega} \int_{\Omega}  \frac{|u(x)-u(y)|^{p}}{|x-y|^{d+sp}} \, dx \, dy < \infty \right\},
\end{equation*}
with the norm
\begin{equation}\label{norm defn}
    \norm{u}_{W^{s,p}(\Omega)} := \left(  [u]^{p}_{W^{s,p}(\Omega)}  + \norm{u}^{p}_{L^{p}(\Omega)}  \right)^{\frac{1}{p}},
\end{equation}
where $[u]_{W^{s,p}(\Omega)}$ denotes the Gagliardo seminorm:
\begin{equation*}
[u]_{W^{s,p}(\Omega)} = \left( \int_{\Omega} \int_{\Omega} \frac{|u(x)-u(y)|^{p}}{|x-y|^{d+sp}} \, dx \, dy \right)^{\frac{1}{p}}.
\end{equation*}
We denote by $W^{s,p}_{0}(\Omega)$ the closure of $C^{\infty}_{c}(\Omega)$ with respect to the fractional Sobolev norm $ \| \cdot \|_{W^{s,p}(\Omega)}$.

\smallskip

When $sp = d$ with $d \geq 2$, and $\Omega$ is a bounded Lipschitz domain, the fractional Sobolev inequality states that for any $\tau \geq p$, there exists a constant $C = C(d,s,\Omega) > 0$ such that (see the Lemma $2.2$ in \cite{Adi2025})
    \begin{equation}\label{frac Sob ineq sp=d}
        \| u \|_{L^{\tau}(\Omega)} \leq C \tau^{\frac{d-s}{d}} [u]_{W^{s,p}(\Omega)}, \hspace{4mm} \forall \ u \in W^{s,p}_{0}(\Omega).
    \end{equation}

\smallskip

Now we state the fractional Sobolev inequality for the case $sp = d$ when the bounded Lipschitz domain is scaled by a parameter $\lambda > 0$, and a function is subtracted by its average. An important observation is that the constant in the inequality is independent of $\lambda$. This lemma is helpful in establishing Theorem \ref{Th: Fractional Hardy sp=k}.

\begin{lemma}\label{sobolev}
   Let $\Omega$ be a bounded Lipschitz domain in $\mathbb{R}^{d}$. Let $p > 1$ and $s \in (0,1)$ be such that $sp = d$. Define $\Omega_{\lambda}: = \left\{ \lambda x : ~ x \in \Omega \right\}$ for $\lambda>0$, then there exists a positive constant $C=C(d,s,\Omega) $ such that, for any $\tau \geq p$, we have
    \begin{equation}
         \left( \fint_{\Omega_{\lambda}} |u(x)-(u)_{\Omega_{\lambda}}|^{\tau } dx \right)^{\frac{1}{\tau}}  \leq C  \tau^{\frac{d-s}{d}} [u]_{W^{s,p}(\Omega_{\lambda})}, \hspace{3mm} \forall \ u \in W^{s,p}(\Omega_{\lambda}).
    \end{equation}
    
\end{lemma} 
\begin{proof}
 See \cite[Lemma $2.1$]{Adi2025} for the proof.
\end{proof}

\smallskip

The next lemma establishes an inequality for the average of $u$ over two disjoint sets. This lemma is helpful in establishing a fractional Hardy inequality with singularity on flat submanifold in Lemma \ref{Le : flat case sp=d} for the case $sp=d$. The proof of the following lemma is available in \cite[Lemma $2.2$]{Adi2025}.

\begin{lemma}\label{avg}
    Let $E$ and $F$ be disjoint sets in $\mathbb{R}^d$. Then for any $\tau \geq 1$, we have
    \begin{equation}
        |(u)_{E} - (u)_{F}|^{\tau} \leq 2^{\tau} \frac{|E \cup F|}{\min \{ |E|, |F| \} }  \fint_{E \cup F} |u(x)-(u)_{E \cup F}|^{\tau}dx  . 
    \end{equation}
\end{lemma}

\smallskip

The following lemma states a basic inequality, which can be found in several articles (see \cite{Adi2025, adimurthiSubmanifold,  kijaczko2025, Squassina2019, squassina2018, Vivek2025}). For a proof, we refer to \cite[Lemma $2.5$]{adimurthiSubmanifold}.

\begin{lemma}\label{estimate}
    Let  $\tau >1$ and $c>1$. Then for all  $a, ~ b \in \mathbb{R}$, we have 
    \begin{equation}
        (|a| + |b|)^{\tau} \leq c|a|^{\tau} + (1-c^{\frac{-1}{\tau -1}})^{1-\tau} |b|^{\tau} .
    \end{equation}
\end{lemma}

\smallskip

The following lemma establishes an inequality for the product of a function $u \in W^{s,p}(\Omega)$ and a function $\xi \in C^{0,1}(\Omega)$. This result plays a crucial role in the patching argument used in the proof of Theorem \ref{Th: Fractional Hardy sp=k}.

\begin{lemma}\label{testfunc}
    Let $\Omega$ be a bounded Lipschitz domain in $\mathbb{R}^d$, and let $sp>1$. Let $u \in W^{s,p}_{0}(\Omega)$ and $\xi \in C^{0,1}(\Omega)$ satisfy $0 \leq \xi \leq 1$. Then there exists a constant $C=C(d,s,p,\Omega)>0$ such that 
    \begin{equation*}
         [\xi u]_{W^{s, p}(\Omega)} \leq C [u]_{W^{s, p}(\Omega)}.
    \end{equation*}
\end{lemma}
\begin{proof}
  Let $u \in W^{s,p}_{0}(\Omega)$ and $\xi \in C^{0,1}(\Omega)$ satisfy $0 \leq \xi \leq 1$. By \cite[Lemma 5.3]{di2012hitchhikers}, we have
\begin{equation*}
       [\xi u]^{p}_{W^{s,p}(\Omega)}  \leq C \left( [u]^{p}_{W^{s,p}(\Omega)}  + \norm{u}^{p}_{L^{p}(\Omega)} \right).
\end{equation*}
Using the fractional Hardy inequality with boundary singularity for the case $sp>1$ (see \cite[Theorem $1.1$ (T$1$)]{dyda2004}) and the fact that $\delta_{\Omega}(x) < C_{\Omega}$ for all $x \in \Omega$, for some $C_{\Omega}>0$,  we obtain
\begin{equation*}
        \int_{\Omega} |u(x)|^{p} \, dx \leq C_{\Omega}^{sp} \int_{\Omega} \frac{|u(x)|^{p}}{\delta_{\Omega}^{sp}(x)} \, dx \leq C [u]^{p}_{W^{s,p}(\Omega)},
\end{equation*}
where $C=C(d,s,p,\Omega)>0$. Therefore, combining the above two inequalities, we obtain
\begin{equation*}
         [\xi u]_{W^{s,p}(\Omega)} \leq C [u]_{W^{s,p}(\Omega)}.
\end{equation*}
This completes the proof.
\end{proof}


\section{Geometry of smooth compact Submanifolds of Codimension \texorpdfstring{$k$}{k}}\label{Geometry of Compact Lipschitz Submanifolds}
In this section, we describe the geometry of a smooth compact set $K$ of codimension $k$ using the definition given in Definition~\ref{definition}. Similar to a bounded Lipschitz domain, it is sufficient to understand $K$ locally in terms of a Lipschitz function and the regions lying above and below its graph. We begin with a simple model example to illustrate this.

\smallskip

\subsection{The case \texorpdfstring{$\delta_{K}(x) \sim |\xi_k|$}{δK(x) ~ |ξk|}} Let $\gamma : \mathbb{R}^{d-k} \to \mathbb{R}^{k}$ be a Lipschitz function such that
\begin{equation*}
  |\gamma(x_{d-k}) - \gamma(y_{d-k})| \leq C_L |x_{d-k} - y_{d-k}|,  
\end{equation*}
for all $x_{d-k},\, y_{d-k} \in \mathbb{R}^{d-k}$, and for some $C_{L}>0$. Define the domain (see Figure \ref{fig:myfigure0})
\begin{equation*}
   \Omega := \left\{ x =(x_{k}, x_{d-k}) \in \mathbb{R}^{d} : x_{k} \neq \gamma(x_{d-k}) \right\}. 
\end{equation*}
This means that for any $x \in \Omega$, there exists a component $x_i$ of $\gamma(x_{d-k})_i$, for some $1 \le i \le k$, such that $x_i \neq \gamma(x_{d-k})_i$. Set 
\begin{equation*}
    K := \partial \Omega = \left\{ x \in \mathbb{R}^{d} :  x_{k} = \gamma(x_{d-k}) \right\}.
\end{equation*}
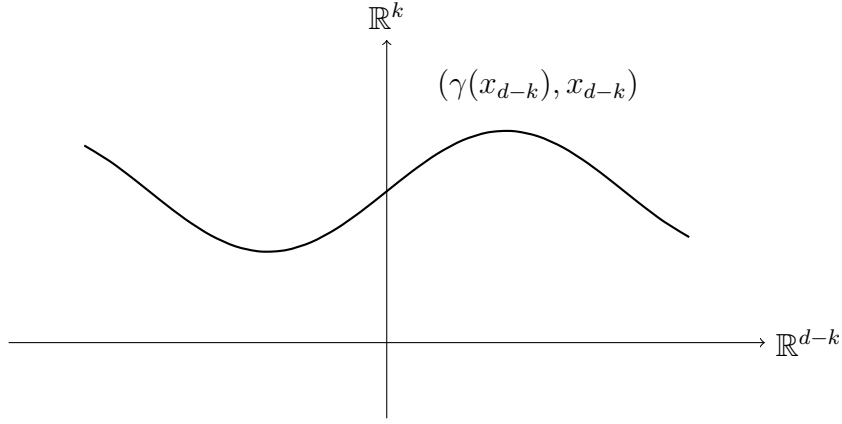
\begin{figure}[h!]
\begin{center}
\begin{tikzpicture}[scale=1.0]
\node at (2,3.4) {$(\gamma(x_{d-k}), x_{d-k})$};

\draw[->] (-5,0) -- (5,0) node[right] {$\mathbb{R}^{d-k}$};
\draw[->] (0,-1) -- (0,4) node[above] {$\mathbb{R}^{k}$};

\draw[thick,domain=-4:4,smooth,variable=\x]
  plot (\x,{2 + 0.8*sin(\x r)}) node[pos=1.8,above right] {};

\end{tikzpicture}
\end{center}
\caption{The domain $\Omega$ and the graph of the map $\gamma$.}
\label{fig:myfigure0}
\end{figure}

Define a map $T : \mathbb{R}^{d} \to \mathbb{R}^{d}$ by
\begin{equation*}
   T(x) = (x_k - \gamma(x_{d-k}),\, x_{d-k}). 
\end{equation*}
Then $T(K) = \left\{ x =(x_{k}, x_{d-k}) \in \mathbb{R}^{d} : x_{k}=0 \right\}$. We now show that $T$ is a Lipschitz map.
\begin{align*}
|T(x)-T(y)|^{2}
&= |x_k - \gamma(x_{d-k}) - y_k + \gamma(y_{d-k})|^{2}
   + |x_{d-k}-y_{d-k}|^{2} \\
&\leq |x_k-y_k|^{2}
   + |\gamma(x_{d-k})-\gamma(y_{d-k})|^{2}
   + 2 \langle x_k-y_k,\gamma(x_{d-k})-\gamma(y_{d-k})\rangle \\
&\qquad + |x_{d-k}-y_{d-k}|^{2} \\ & \leq (C^{2}_{L}+1)|x-y|^{2} + 2 \langle x_k-y_k,\gamma(x_{d-k})-\gamma(y_{d-k})\rangle .
\end{align*}
Using $2\langle a,b\rangle \leq |a|^{2} + |b|^{2}$, we obtain
\begin{align*}
|T(x)-T(y)|^{2}
&\leq (C_{L}^{2}+1)|x - y|^{2}
   + |x_k-y_k|^{2}
   + |\gamma(x_{d-k})-\gamma(y_{d-k})|^{2} \\
&\leq (2 C_{L}^{2}+2)\, |x-y|^{2}
   = C_{1}|x-y|^{2}.
\end{align*}
For the inverse map $T^{-1}(x) = (x_{k}+ \gamma(x_{d-k}),\, x_{d-k})$, the same analysis yields
\begin{equation*}
    |T^{-1}(x)-T^{-1}(y)|^{2} \leq C_{1}|x-y|^{2}.
\end{equation*}
Therefore,
\begin{equation*}
    \frac{1}{C} |x-y| \leq |T(x)-T(y)| \leq C |x-y|, \quad \forall \ x, y \in \mathbb{R}^{d}, 
\end{equation*}
for some $C>0$.

\smallskip

Now, let $x \in \Omega$ and $y \in K$ be such that
\begin{equation*}
    \delta_{K}(x) = |x-y| = \inf_{\eta \in K} |x- \eta|.
\end{equation*}
Then
\begin{equation*}
    \delta_{K}(x) = |x-y| \le |x- \eta|,
\quad\text{for all }\eta \in K.
\end{equation*}
Therefore,
\begin{equation*}
    \delta_{K}(x)
= |x-y|
\le C |T(x)-T(\eta)|
= C |T(x) - z|,
\end{equation*}
for all $z \in T(K) = \{ x \in \mathbb{R}^{d} : x_{k}=0 \}$. Thus,
\begin{equation*}
   \delta_{K}(x) \le C \inf_{z \in T(K)} |T(x)-z|
= C |\xi_{k}|, 
\end{equation*}
where $T(x) = (\xi_k, \xi_{d-k})$.  
Using the Lipschitz property of $T^{-1}$, we also obtain $C|\xi_k| \le \delta_{K}(x)$.  
Therefore,
\begin{equation*}
   \delta_{K}(x) \sim |\xi_k|. 
\end{equation*}
This shows that the condition stated in Definition \ref{definition} is not an additional assumption, but rather a geometric consequence of the local coordinate representation of $K$.

\smallskip

\subsection{Jacobian Properties of the Flattening Map}\label{Jacobian of a Map}

We now verify that the mapping $T$ and its inverse $T^{-1}$ both have Jacobian
determinant equal to $1$. Recall that
\[
T(x) = (x_k - \gamma(x_{d-k}),\, x_{d-k})
\qquad\text{and}\qquad
T^{-1}(x) = (x_k + \gamma(x_{d-k}),\, x_{d-k}),
\]
where $x = (x_k, x_{d-k}) \in \mathbb{R}^{k} \times \mathbb{R}^{d-k}$ and
$\gamma : \mathbb{R}^{d-k} \to \mathbb{R}^{k}$ is a Lipschitz function.
Since $\gamma$ is Lipschitz, it is differentiable a.e.\ by Rademacher's theorem,
and the following Jacobian computations hold at every such point.

\medskip
\noindent\textbf{Jacobian of $T$.}
Writing $T(x) = (T_1(x), T_2(x))$, we have
\[
T_1(x) = x_k - \gamma(x_{d-k}), 
\qquad
T_2(x) = x_{d-k}.
\]
Differentiating with respect to the two blocks $x_k$ and $x_{d-k}$ gives
\[
DT(x) =
\begin{pmatrix}
I_k & -D\gamma(x_{d-k}) \\
0   & I_{d-k}
\end{pmatrix},
\]
where $D\gamma(x_{d-k})$ denotes the $k \times (d-k)$ Jacobian matrix of $\gamma$.
This matrix is block upper triangular, and therefore
\[
\det DT(x) = \det(I_k)\,\det(I_{d-k}) = 1.
\]

\medskip
\noindent\textbf{Jacobian of $T^{-1}$.}
Writing $T^{-1}(x) = (S_1(x), S_2(x))$ with
\[
S_1(x) = x_k + \gamma(x_{d-k}), 
\qquad 
S_2(x) = x_{d-k},
\]
we similarly obtain
\[
DT^{-1}(x) =
\begin{pmatrix}
I_k & D\gamma(x_{d-k}) \\
0   & I_{d-k}
\end{pmatrix}.
\]
Again, this is block upper triangular with identity blocks on the diagonal, so
\[
\det DT^{-1}(x) = \det(I_k)\,\det(I_{d-k}) = 1.
\]
Therefore, at every point where $\gamma$ is differentiable, both $T$ and $T^{-1}$
have Jacobian determinant equal to $1$.

\section{Fractional Hardy inequality with singularity on submanifold}\label{fractional hardy}

In this section, we establish a fractional Hardy inequality with singularity along a smooth submanifold of codimension $k$, where $1<k<d$ in the case $sp=d$ and $\tau \geq p$. For any $x \in \mathbb{R}^{d} $, we write $x=(x_{k}, x_{d-k})$, where $x_{k} \in \mathbb{R}^{k}$ and $x_{d-k} \in \mathbb{R}^{d-k}$. We first prove the Hardy inequality on the domain $Q$, as defined in Definition \ref{definition}. Then, using a partition of unity and standard patching techniques, we extend the inequality to smooth submanifolds. Recall  
\begin{equation*}
    Q = \left\{ x= (x_{k}, x_{d-k}) \in \mathbb{R}^{d} :  |x_{k}| < 1 \ \text{and} \ x_{d-k} \in (0, 1)^{d-k}   \right\}.
\end{equation*}
For each  $\ell \leq - 1$, we define (see Figure \ref{fig:myfigure2})
\begin{equation*}
    A_{\ell} = \{ x=(x_{k},x_{d-k}) \in Q :  2^{ \ell} \leq |x_{k}| < 2^{\ell +1} \}.
\end{equation*}
\begin{figure}[!ht] 
\begin{center}
\begin{tikzpicture}
\node at (-2,.5) {$2^{\ell} \leq |x_{k}| < 2^{\ell +1}$};
\node at (1.5,1.5) {$(0,1)^{d-k}$};

\draw[dashed] (0,1) -- (3,1);
\draw[dashed] (0,-1) -- (3,-1);

\draw (0,0) ellipse [y radius=1, x radius=0.5];
\draw (3,0) ellipse [y radius=1, x radius=0.5];
\fill[blue!20, opacity=0.5] (0,0) ellipse [y radius=1, x radius=0.5];
\fill[blue!20, opacity=0.5] (3,0) ellipse [y radius=1, x radius=0.5];

\fill[white, opacity=1] (0,0) ellipse [y radius=0.5, x radius=0.25];
\fill[white, opacity=1] (3,0) ellipse [y radius=0.5, x radius=0.25];

\draw[opacity=0.5] (0,0) ellipse [y radius=.5, x radius=0.25];
\draw[opacity=0.5] (3,0) ellipse [y radius=.5, x radius=0.25];

\foreach \x in {0,0.2,...,3}
\draw[blue!20,opacity=0.5] (\x,0) ellipse [y radius=1, x radius=0.5];
    
\foreach \x in {0,0.2,...,3}
\draw[blue!20,opacity=0.5] (\x,0) ellipse [y radius=.5, x radius=0.25];
    
\draw[->] (-1,0) -- (4,0) node[right] {$\mathbb{R}^{d-k}$};
\draw[->] (0,-1.75) -- (0,1.75) node[above] {$\mathbb{R}^{k}$};
\end{tikzpicture}
\end{center}
\caption{$A_{\ell}$.}
\label{fig:myfigure2}
\end{figure}

\noindent Therefore, we have
\begin{equation*}
    Q \setminus K = \bigcup_{\ell= - \infty}^{-1} A_{\ell},
\end{equation*}
where $K = \{ x=(x_{k}, x_{d-k}) \in Q : x_{k} =0 \}$. Again, we further divide  $A_{\ell}$ into a disjoint union of identical sets, denoted as  $A^{i}_{\ell}$ (see Figure \ref{fig:myfigure3}), such that if  $x = (x_{k}, x_{d-k}) \in A^{i}_{\ell}$, then  $x_{d-k} \in C^{i}_{\ell}$, where  $C^{i}_{\ell}$  is a cube of side length  $2^{ \ell}$ in $\mathbb{R}^{d-k}$. Then, we have
\begin{equation*}
      A_{\ell} = \bigcup_{i = 1}^{\sigma_{\ell}} A^{i}_{\ell}  ,
\end{equation*}
where  $\sigma_{\ell} = 2^{(- \ell)(d-k)}$. The Lebesgue measure of  $A^{i}_{\ell}$ is then given by
\begin{equation*}
    |A^{i}_{\ell}| = \left( \frac{|\mathbb{S}^{k-1}|(2^{k}-1)2^{\ell k}}{k} \right) \times 2^{\ell (d-k)} =  \frac{|\mathbb{S}^{k-1}|(2^{k}-1)2^{\ell d}}{k}  ,
\end{equation*}
where $|\mathbb{S}^{k-1}|$ denotes the $(k-1)$-dimensional surface measure of the unit sphere in $\mathbb{R}^{k}$.
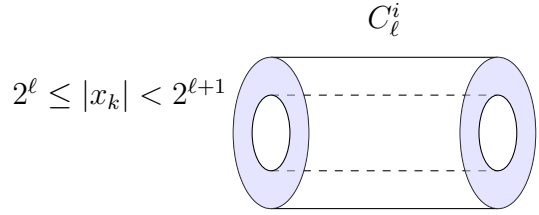
\begin{figure}[!ht] 
\begin{center}
\begin{tikzpicture}
\node at (-2,.5) {$2^{\ell} \leq |x_{k}| < 2^{\ell +1}$};
    
\draw (0,1) -- (3,1);
\draw (0,-1) -- (3,-1);
\draw[dashed] (0,.5) -- (3,.5);
\draw[dashed] (0,-.5) -- (3,-.5);

\draw (0,0) ellipse [y radius=1, x radius=0.5];
\fill[blue!20, opacity=0.5] (0,0) ellipse [y radius=1, x radius=0.5];
\fill[white, opacity=1] (0,0) ellipse [y radius=0.5, x radius=0.25];
\draw (0,0) ellipse [y radius=.5, x radius=0.25];

\draw (3,0) ellipse [y radius=1, x radius=0.5];
\fill[blue!20, opacity=0.5] (3,0) ellipse [y radius=1, x radius=0.5];
\fill[white, opacity=1] (3,0) ellipse [y radius=0.5, x radius=0.25];
\draw (3,0) ellipse [y radius=.5, x radius=0.25];
\node at (1.5,1.5) {$C^{i}_{\ell}$};
\end{tikzpicture}
\end{center}
\caption{$A^{i}_{\ell}= \{ x=(x_{k},x_{d-k}) \in A_{\ell} : x_{d-k} \in C^{i}_{\ell} \}   $.}
\label{fig:myfigure3}
\end{figure}

The next lemma establishes a fractional Hardy inequality on the domain $Q$ with singularity along the flat submanifold $K = \{ x=(x_{k}, x_{d-k}) \in Q : x_{k} =0 \}$. This lemma, together with Definition \ref{definition}, will be helpful in proving the fractional Hardy inequality with singularity on a smooth submanifold of codimension $k$.

\begin{lemma}\label{Le : flat case sp=d}
     Let $Q=B^{k}_{1}(0) \times (0,1)^{d-k}$ and define $K = \{ x=(x_{k}, x_{d-k}) \in Q : x_{k} =0 \}$, where $k \in \mathbb{N}$ and $1<k<d$. Let $p>1$ and $s \in (0,1)$ be such that $sp=d$. Then for any  $\tau \geq p$, there exists a constant  $C= C(d,s,k) > 0$  such that
    \begin{equation}
     \left(  \bigintssss_{Q} \frac{|u(x)|^{\tau}}{|x_{k}|^{d} } \,  dx \right)^{\frac{1}{\tau}} \leq  C \tau^{\frac{d-s}{d}}
 [u]_{W^{s, p}(Q)}, \quad \forall \ u \in W^{s,p}_{0}(Q \setminus K).
    \end{equation}
\end{lemma}
\begin{proof}
 Since $C^{\infty}_{c}(\Omega \setminus K)$ is dense in $W^{s,p}_{0}(\Omega \setminus K)$, it is sufficient to prove the inequality for functions $u \in C^{\infty}_{c}(\Omega \setminus K)$. For any $A^{i}_{\ell}$, applying Lemma \ref{sobolev} with $\Omega = \{ (x_{k}, x_{d-k}) : 1 < |x_{k}| < 2  \ \text{and} \ x_{d-k} \in (1,2)^{d-k} \}$, and $ \lambda = 2^{\ell}$ and using translation invariance, we have
\begin{equation*}
    \fint_{A^{i}_{\ell}} |u(x)-(u)_{A^{i}_{\ell}}|^{\tau} \,  dx \leq C^{\tau}  \tau^{\frac{(d-s) \tau}{d}} [u]^{\tau}_{W^{s, p}(A^{i}_{\ell})}  ,
\end{equation*}
where $C= C(d,s,k)>0$. For any $x = (x_{k},x_{d-k}) \in A^{i}_{\ell}$, we have $ \frac{1}{|x_{k}|} \leq \frac{1}{2^{\ell}}$. Therefore, using this we arrive at
\begin{align*}
    \int_{A^{i}_{\ell}} \frac{|u(x)|^{\tau}}{|x_{k}|^{d}} \,  dx & \leq \frac{1}{2^{\ell d}} \int_{A^{i}_{\ell}} |u(x)-(u)_{A^{i}_{\ell}} + (u)_{A^{i}_{\ell}}|^{\tau} \,  dx \\ &
    \leq \frac{2^{\tau}}{2^{\ell d}} \int_{A^{i}_{\ell}} |u(x)-(u)_{A^{i}_{\ell}}|^{\tau} \,  dx + \frac{2^{\tau}}{2^{\ell d}} \int_{A^{i}_{\ell}} |(u)_{A^{i}_{\ell}}|^{\tau} \, dx.
    \end{align*}
  Now, using the fractional Sobolev inequality for any $A^{i}_{\ell}$ as mentioned above,  we obtain
    \begin{align*}
    \int_{A^{i}_{\ell}} \frac{|u(x)|^{\tau}}{|x_{k}|^{d}} \,  dx & \leq 2^{\tau} \frac{|A^{i}_{\ell}|}{2^{\ell d}} \fint_{A^{i}_{\ell}} |u(x)-(u)_{A^{i}_{\ell}}|^{\tau} \,  dx + 2^{\tau} \frac{|A^{i}_{\ell}|}{2^{\ell d}} |(u)_{A^{i}_{\ell}}|^{\tau} \\ &
    \leq C^{\tau} \tau^{\frac{(d-s)\tau}{d}} \left( \frac{|\mathbb{S}^{k-1}|(2^{k}-1)2^{\ell d}}{k} \right) \frac{1}{2^{\ell d}}  [u]^{\tau}_{W^{s, p}(A^{i}_{\ell})} \\ & \hspace{5mm}  + 2^{\tau} \left( \frac{|\mathbb{S}^{k-1}|(2^{k}-1)}{k} \right)  |(u)_{A^{i}_{\ell}}|^{\tau}  
   \\ & \leq C^{\tau} \tau^{\frac{(d-s)\tau}{d}} [u]^{\tau}_{W^{s, p}(A^{i}_{\ell})} + C^{\tau}  |(u)_{A^{i}_{\ell}}|^{\tau}  ,
\end{align*}   
where $C= C(d,s, k)$ is a constant. By summing the above inequality from $i=1$ to $\sigma_{\ell}$, we obtain
\begin{equation*}
 \sum_{i=1}^{\sigma_{\ell}} \int_{A^{i}_{\ell}} \frac{|u(x)|^{\tau}}{|x_{k}|^{d}} \,  dx \leq   C^{\tau} \tau^{\frac{(d-s)\tau}{d}} \sum_{i=1}^{\sigma_{\ell}}  [u]^{\tau}_{W^{s, p}(A^{i}_{\ell})} + C^{\tau} \sum_{i=1}^{\sigma_{\ell}}  |(u)_{A^{i}_{\ell}}|^{\tau}.
\end{equation*}
From the inequality \eqref{sumineq} with $\gamma = \frac{\tau}{p} \geq 1$, we have
\begin{equation*}
    \sum_{i=1}^{\sigma_{\ell}} \left( [u]^{p}_{W^{s, p}(A^{i}_{\ell})} \right)^{\frac{\tau}{p}} \leq \left( \sum_{i=1}^{\sigma_{\ell}} [u]^{p }_{W^{s, p}(A^{i}_{\ell})} \right)^{\frac{\tau}{p}} \leq [u]^{\tau}_{W^{s,p}(A_{\ell})}.
\end{equation*}
Combining the above two inequalities, we obtain
\begin{equation*}
     \int_{A_{\ell}} \frac{|u(x)|^{\tau}}{|x_{k}|^{d}}  \, dx \leq   C^{\tau} \tau^{\frac{(d-s)\tau}{d}}  [u]^{\tau}_{W^{s, p}(A_{\ell})} + C^{\tau} \sum_{i=1}^{\sigma_{\ell}}  |(u)_{A^{i}_{\ell}}|^{\tau}.
\end{equation*}
Again, summing the above inequality from $\ell = m$ to $-1$, we arrive at 
\begin{equation*}
  \sum_{\ell=m}^{-1} \int_{A_{\ell}} \frac{|u(x)|^{\tau}}{|x_{k}|^{d}} \,  dx \leq   C^{\tau} \tau^{\frac{(d-s)\tau}{d}} \sum_{\ell=m}^{-1}  [u]^{\tau}_{W^{s, p}(A_{\ell})} + C^{\tau} \sum_{\ell=m}^{-1} \sum_{i=1}^{\sigma_{\ell}}  |(u)_{A^{i}_{\ell}}|^{\tau}.  
\end{equation*}
Applying the inequality \eqref{sumineq} with $\gamma = \frac{\tau}{p} \geq 1$, we have
\begin{equation*}
    \sum_{\ell=m}^{-1} \sum_{i=1}^{\sigma_{\ell}}  |(u)_{A^{i}_{\ell}}|^{\tau} =  \sum_{\ell=m}^{-1} \sum_{i=1}^{\sigma_{\ell}}  \left(|(u)_{A^{i}_{\ell}}|^{p} \right)^{ \frac{\tau}{p}} \leq \left( \sum_{\ell=m}^{-1} \sum_{i=1}^{\sigma_{\ell}}  |(u)_{A^{i}_{\ell}}|^{p}\right)^{\frac{\tau}{p}}.
\end{equation*}
Therefore, combining the above two inequalities, we obtain
\begin{equation}\label{eqnn1}
  \sum_{\ell=m}^{-1} \int_{A_{\ell}} \frac{|u(x)|^{\tau}}{|x_{k}|^{d}} \,  dx \leq   C^{\tau} \tau^{\frac{(d-s)\tau}{d}} \sum_{\ell=m}^{-1}  [u]^{\tau}_{W^{s, p}(A_{\ell})} + C^{\tau} \left( \sum_{\ell=m}^{-1} \sum_{i=1}^{\sigma_{\ell}}  |(u)_{A^{i}_{\ell}}|^{p}\right)^{\frac{\tau}{p}}. 
\end{equation}
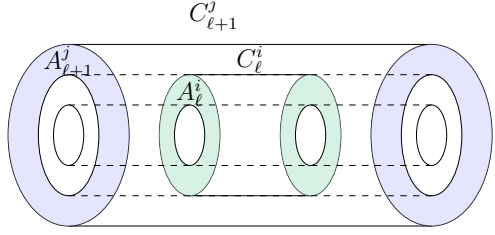
\begin{figure}[!ht]
\begin{center}
\scalebox{0.8}{%
\begin{tikzpicture}

\draw (0,1.5) -- (6,1.5);
\draw (0,-1.5) -- (6,-1.5);

\draw (2,1) -- (4,1);
\draw (2,-1) -- (4,-1);

\draw (0,0) ellipse [y radius=1.5, x radius=1];
\fill[blue!20, opacity=0.5] (0,0) ellipse [y radius=1.5, x radius=1];
\fill[white, opacity=1] (0,0) ellipse [y radius=1, x radius=0.5];
\draw (0,0) ellipse [y radius=.5, x radius=0.25];

\draw (2,0) ellipse [y radius=1, x radius=0.5];
\fill[blue!20, opacity=0.5] (2,0) ellipse [y radius=1, x radius=.5];
\fill[green!20, opacity=0.5] (2,0) ellipse [y radius=1, x radius=.5];
\fill[white, opacity=1] (2,0) ellipse [y radius=.5, x radius=0.25];
\draw (0,0) ellipse [y radius=1, x radius=0.5];
\draw (2,0) ellipse [y radius=.5, x radius=0.25];

\draw (6,0) ellipse [y radius=1.5, x radius=1];
\fill[blue!20, opacity=0.5] (6,0) ellipse [y radius=1.5, x radius=1];
\fill[white, opacity=1] (6,0) ellipse [y radius=1, x radius=0.5];
\draw (6,0) ellipse [y radius=.5, x radius=0.25];

\draw (4,0) ellipse [y radius=1, x radius=0.5];
\fill[blue!20, opacity=0.5] (4,0) ellipse [y radius=1, x radius=.5];
\fill[green!20, opacity=0.5] (4,0) ellipse [y radius=1, x radius=.5];
\fill[white, opacity=1] (4,0) ellipse [y radius=.5, x radius=0.25];
\draw (6,0) ellipse [y radius=1, x radius=0.5];
\draw (4,0) ellipse [y radius=.5, x radius=0.25];

\draw[dashed] (0,1) -- (6,1);
\draw[dashed] (0,-1) -- (6,-1);
\draw[dashed] (0,-.5) -- (6,-.5);
\draw[dashed] (0,.5) -- (6,.5);

\node at (2,.7) {$A^{i}_{\ell}$};
\node at (0,1.22) {$A^{j}_{\ell +1}$};
\node at (2.4,2) {$C^{j}_{\ell+1}$};
\node at (3,1.25) {$C^{i}_{\ell}$};

\end{tikzpicture}%
}
\end{center}
\caption{$A^{i}_{\ell}$ and $A^{j}_{\ell +1}$.}
\label{fig:myfigure4}
\end{figure}
Let  $A^{j}_{\ell+1}$ be a set such that   $A^{i}_{\ell}$ lies below the set  $A^{j}_{\ell+1}$ (see Figure \ref{fig:myfigure4}). Then using Lemma  \ref{avg} with  $E=A^{i}_{\ell}$ and  $F = A^{j}_{\ell+1}$, we get 
\begin{equation}\label{compineq1}
|(u)_{A^{i}_{\ell}} - (u)_{A^{j}_{\ell+1}} |^{p} \leq C \fint_{A^{i}_{\ell} \cup A^{j}_{\ell+1} }  |u(x) - (u)_{A^{i}_{\ell} \cup A^{j}_{\ell+1}}|^{p} \, dx,
\end{equation}
where $C= C(d,s,k)$ is a constant. Choose an open set  $\Omega$ such that  $\Omega_{\lambda}$ is a translation of  $ A^{i}_{\ell} \cup A^{j}_{\ell+1}$ with scaling parameter  $\lambda=2^{\ell+1}$. Applying Lemma  \ref{sobolev} with this  $\Omega$ and  $\lambda=2^{\ell+1}$, and using translation invariance, we obtain
\begin{equation}\label{compineq2}
   |(u)_{A^{i}_{\ell}} - (u)_{A^{j}_{\ell+1}} |^{p} \leq C \fint_{A^{i}_{\ell} \cup A^{j}_{\ell+1} }  |u(x) - (u)_{A^{i}_{\ell} \cup A^{j}_{\ell+1}}|^{p} \, dx \leq  C [u]^{p}_{W^{s, p}(A^{i}_{\ell} \cup A^{j}_{\ell+1})},
\end{equation}
where $C=C(d,s,k)>0$. Using the triangle inequality, we have 
\begin{equation*}
    |(u)_{A^{j}_{\ell+1}}|^{p} \leq  \left( |(u)_{A^{i}_{\ell}}| + |(u)_{A^{i}_{\ell}} - (u)_{A^{j}_{\ell+1}}| \right)^{p}.
\end{equation*}
Applying Lemma \ref{estimate} with $c :=c_{1}2^{d-k}>1$ where $c_{1} = \frac{2}{1+ 2^{d-k}} <1$ together with the inequality \eqref{compineq2}, we obtain
\begin{equation*}
    |(u)_{A^{j}_{\ell+1}}|^{p} \leq c_{1} 2^{d-k} |(u)_{A^{i}_{\ell}}|^{p} +  C [u]^{p}_{W^{s, p}(A^{i}_{\ell} \cup A^{j}_{\ell+1})}  ,
\end{equation*}
where $C=C(d,s,k)>0$. Since there are $2^{d-k}$ such $A^{i}_{\ell}$'s lies below $A^{j}_{\ell+1}$,  summing the above inequality from $i=2^{d-k}(j-1)+1$ to  $2^{d-k}j$, we obtain
\begin{align*}
     2^{d-k}  |(u)_{A^{j}_{\ell+1}}|^{p}  \leq c_{1} 2^{ d-k} \sum_{i=2^{d-k}(j-1)+1}^{2^{d-k}j}  |(u)_{A^{i}_{\ell}}|^{p} 
       +  C  \sum_{i=2^{d-k}(j-1)+1}^{2^{d-k}j} [u]^{p}_{W^{s, p}(A^{i}_{\ell} \cup A^{j}_{\ell+1})}  .
\end{align*}
Multiplying the above inequality with $\frac{1}{2^{d-k}}$ and using $\frac{1}{2^{d-k}} \leq 1$, and summing the above inequality from $j=1$ to $\sigma_{\ell+1}$, we obtain
\begin{align}\label{ineqn1}
  \sum_{j=1}^{\sigma_{\ell+1}}  |(u)_{A^{j}_{\ell+1}}|^{p} & \leq c_{1}
\sum_{i=1}^{\sigma_{\ell}} |(u)_{A^{i}_{\ell}}|^{p} 
  +  C \sum_{j=1}^{\sigma_{\ell+1}} \left( \sum_{i=2^{d-k}(j-1)+1}^{2^{d-k}j}  [u]^{p}_{W^{s, p}(A^{i}_{\ell} \cup A^{j}_{\ell+1})} \right) \nonumber \\ & \leq c_{1}  
\sum_{i=1}^{\sigma_{\ell}} |(u)_{A^{i}_{\ell}}|^{p} + C  [u]^{p}_{W^{s, p}(A_{\ell} \cup A_{\ell+1})}  . 
\end{align}
By summing the above inequality from $\ell=m \in \mathbb{Z}^{-}$ to $-2$, we get
\begin{equation*}
   \sum_{\ell=m}^{-2}   \sum_{j=1}^{\sigma_{\ell+1}}  |(u)_{A^{j}_{\ell+1}}|^{p} \leq c_{1} \sum_{\ell=m}^{-2} \sum_{i=1}^{\sigma_{\ell}} |(u)_{A^{i}_{\ell}}|^{p} + C   \sum_{\ell=m}^{-2} [u]^{p}_{W^{s, p}(A_{\ell} \cup A_{\ell+1})}  .
\end{equation*}
By changing sides, rearranging, and re-indexing, we get
\begin{equation*}
 (1-c_{1})   \sum_{\ell=m+1}^{-1} \sum_{i=1}^{\sigma_{\ell}} |(u)_{A^{i}_{\ell}}|^{p} \leq  \sum_{j=1}^{\sigma_{m}} |(u)_{A^{j}_{m}}|^{p} + C\sum_{\ell=m}^{-2} [u]^{p}_{W^{s, p}(A_{\ell} \cup A_{\ell+1})}  .
\end{equation*}
Now choose $-m$ large enough so that $|(u)_{A^{j}_{m}}|=0$ for all $j \in \{ 1, \dots, \sigma_{m} \}$. Then the above inequality reduces to
\begin{equation*}
       \sum_{\ell=m}^{-1}  \sum_{i=1}^{\sigma_{\ell}} |(u)_{A^{i}_{\ell}}|^{p} \leq C  \sum_{\ell=m}^{-2} [u]^{p}_{W^{s, p}(A_{\ell} \cup A_{\ell+1})} \leq C[u]^{p}_{W^{s, p}(Q)}  ,
\end{equation*}
where $C=C(d,s,k)>0$. Therefore, we have
\begin{equation}\label{eqnn}
      \left( \sum_{\ell=m}^{-1}  \sum_{i=1}^{\sigma_{\ell}} |(u)_{A^{i}_{\ell}}|^{p} \right)^{\frac{\tau}{p}} \leq   C^{\tau} [u]^{\tau}_{W^{s, p}(Q)} .
\end{equation}
Combining the inequalities \eqref{eqnn1} and \eqref{eqnn} and using $\tau^{\frac{(d-s) \tau}{d}} \geq 1$ yields
 \begin{equation*}
      \left( \int_{Q} \frac{|u(x)|^{\tau}}{|x_{k}|^{d} } \, dx \right)^{\frac{1}{\tau}} \leq C  \tau^{\frac{d-s}{d}}[u]_{W^{s, p}(Q)}.
 \end{equation*}
 This proves the lemma.
\end{proof}

\smallskip

\begin{proof}[\textbf{Proof of Theorem \ref{Th: Fractional Hardy sp=k}}]
Let $\Omega$ be a bounded Lipschitz domain in $\mathbb{R}^{d}$, and let $K \subset \Omega$ be a compact set of class $C^{0,1}$ of  codimension $k$, where $k \in \mathbb{N}$ and $1<k<d$. For any $x \in K$, there exists a ball $B_{r_{x}}(x), ~ r_{x}>0$ such that Definition  \ref{definition} hold with an isomorphism $T_{x}$. Then $ K \subset \cup_{x \in  K} B_{r_{x}}(x)$ and since $K$ is compact, there exists $x_{1}, \dots, x_{n} \in  K$ such that
\begin{equation*}
    K  \subset \bigcup_{i=1}^{n} B_{r_{i}}(x_{i}) ,
\end{equation*}
where $r_{i}=r_{x_{i}}$. Since $C^{\infty}_{c}(\Omega \setminus K)$ is dense in $W^{s,p}_{0}(\Omega \setminus K)$, it is sufficient to prove the inequality for functions $u \in C^{\infty}_{c}(\Omega \setminus K)$. Let $u \in C^{\infty}_{c}(\Omega \setminus K)$ and $sp=d$. Let  $\Omega \setminus K \subset \cup_{i=0}^{n} \Omega_{i}$, where  $\Omega_{0} \subset \Omega \setminus K$ is a bounded Lipschitz domain in $\mathbb{R}^{d}$ such that $\delta_{K}(x)> R$ for all $x \in \Omega_{0}$ for some $R>0$, and $\Omega_{i} = B_{r_{i}}(x_{i})$ for all  $ 1 \leq i \leq n $. Let $ \{ \eta_{i} \}_{i=0}^{n}$ be the associated partition of unity. Then
\begin{equation*}
    u = \sum_{i=0}^{n} u_{i}, \hspace{.3cm} \text{where} \  u_{i} = \eta_{i} u.
\end{equation*}
From Lemma \ref{testfunc}, we have
 \begin{equation*}
    [ u_{i}]_{W^{s,p}(\Omega)} \leq C [u ]_{W^{s,p}(\Omega)}, \hspace{3mm} \forall \ 0 \leq i \leq n .
\end{equation*}
Therefore, it is sufficient to prove Theorem \ref{Th: Fractional Hardy sp=k} for all  $u_{i}, ~ 0 \leq i \leq n$. As $\operatorname{supp} u_{0}  \subset \Omega_{0}$, and for all $x \in \Omega_{0}$, we have
 \begin{equation*}
     C_{1} \leq \delta_{K}(x) \leq C_{2}, \quad \text{for some} \  C_{1}, C_{2} >0 .
 \end{equation*}
 Therefore, using the above estimate for any $x \in \Omega_{0}$, and  the fractional Sobolev inequality \eqref{frac Sob ineq sp=d}, we get
\begin{equation*} 
 \left( \int_{\Omega_{0}} \frac{|u_{0}(x)|^{\tau}}{\delta_{K}^{d}(x) } \,  dx \right)^{\frac{1}{\tau}} \leq C \Bigg( \int_{\Omega_{0}} |u_{0}(x)|^{\tau} \, dx \Bigg)^{\frac{1}{\tau}} \leq C \tau^{\frac{d-s}{d}} [u_{0}]_{W^{s,p}(\Omega_{0})} .
\end{equation*}
For any $1 \leq i \leq n$, we have $ \operatorname{supp} u_{i}  \subset (\Omega \setminus K ) \cap \Omega_{i} $. For any $x_{i} \in K$, consider the isomorphism $T_{x_{i}}$ which is obtained using Definition \ref{definition} and Section \ref{Geometry of Compact Lipschitz Submanifolds}, then 
 \begin{equation*}
      \delta_{K}(x) \sim |\xi_{k}| \hspace{.3cm} \text{for all} \  x \in (\Omega \setminus K) \cap \Omega_{i},
 \end{equation*}
where $T_{x_{i}}((\xi_{k}, \xi_{d-k})) = x$.
Since $T_{x_i}$ is a bi-Lipschitz homeomorphism onto its image and \begin{equation*} \operatorname{supp} u_i \subset (\Omega\setminus K)\cap\Omega_i, \end{equation*} we have \begin{equation*} \operatorname{supp}(u_i\circ T_{x_i}) = T_{x_i}^{-1}(\operatorname{supp}u_i) \subset T_{x_i}^{-1}\bigl((\Omega\setminus K)\cap\Omega_i\bigr) \subset Q. \end{equation*} Extending the function $u_i\circ T_{x_i}$ on $Q$ by $0$ and hence, $u_i\circ T_{x_i} \in W^{s,p}_{0}(Q \setminus K_{0})$, where $K_{0} = \{ x=(x_{k}, x_{d-k}) \in Q : x_{k} =0 \}$ . Therefore, applying Lemma \ref{Le : flat case sp=d} with the above $u_i\circ T_{x_i}$ and using Subsection \ref{Jacobian of a Map}, we have 
\begin{align*} \left( \int_{(\Omega \setminus K) \cap \Omega_{i}} \frac{|u_{i}(x)|^{\tau}}{\delta_{K}^{d}(x) } dx \right)^{\frac{1}{\tau}} & \sim \left( \int_{T^{-1}_{x_{i}}((\Omega \setminus K ) \cap \Omega_{i})} \frac{|u_{i} \circ T_{x_{i}}(\xi)|^{\tau}}{|\xi_{k}|^{d}} d \xi \right)^{\frac{1}{\tau}} \\ & \leq C \tau^{\frac{d-s}{d}} [ u_{i} \circ T_{x_{i}} ]_{W^{s,p}(T^{-1}_{x_{i}}((\Omega \setminus K) \cap \Omega_{i}))} \\ & \leq \, C \tau^{\frac{d-s}{d}} [ u_{i} ]_{W^{s,p}((\Omega \setminus K) \cap \Omega_{i})}, 
\end{align*}
where $C=C(d,s,k,\Omega,K,T_{x_{i}})>0$. Since there are only finitely many maps $T_{x_{i}}$, we may choose a uniform constant $C=C(d,s,k,\Omega,K)>0$ such that the above inequalities hold. This completes the proof of the theorem.
\end{proof}


\section{Proof of Theorem \ref{Th : Trud-Moser inequ}}\label{Proof of Theorem 1}

In this section, we prove Theorem \ref{Th : Trud-Moser inequ}. First, we establish the existence of $\alpha_{*}$ as stated in Theorem \ref{Th : Trud-Moser inequ}. Then, using a Moser-type function, we show that the inequality fails for any $\alpha > \alpha^{*}_{s,d}$.

\subsection{Existence of \texorpdfstring{$\alpha_{*}$}{alpha not} in Theorem \ref{Th : Trud-Moser inequ}}\label{subsec1}
By Theorem \ref{Th: Fractional Hardy sp=k}, for every $\tau \geq p$ and $u \in W^{s,p}_{0}(\Omega \setminus K)$ such that $[u]_{W^{s,p}(\Omega)} \leq 1$, we have the following inequality:
\begin{equation*}
    \left(  \int_{\Omega} |u(x)|^{\tau} \frac{dx}{\delta^{d}_{K}(x)} \right)^{\frac{1}{\tau}} \leq C\tau^{\frac{d-s}{d}},
\end{equation*}
where $C= C(d,s,k, \Omega, K)$ is a positive constant. Now, for  $\alpha > 0$, using   $\frac{nd}{d-s} \geq p = \frac{d}{s}$ for all  $n \geq j_{p}-1$ as  $j_{p}-1 \geq \frac{d}{s}-1 > j_{p}-2$, and from the definition of  $\Phi_{d,s}$ defined in  \eqref{Defn: Phi}, we obtain
\begin{align*}
      \int_{\Omega} \Phi_{d,s} \left( \alpha  |u(x)|^{\frac{d}{d-s}}  \right) \, \frac{dx}{\delta^{d}_{K}(x)}  & = \sum_{n=j_{p}-1}^{\infty} \frac{\alpha^{n}}{n!} \int_{\Omega}  |u(x)|^{\frac{nd}{d-s}} \, \frac{dx}{\delta^{d}_{K}(x)} \\ & \leq \sum_{n=j_{p}-1}^{\infty} \frac{1}{n!} \left( \alpha C^{\frac{d}{d-s}}\frac{d}{d-s}  \right)^{n} n^{n}.   
\end{align*}
Using Stirling's approximation  $n! \sim \sqrt{2 \pi n} \left( \frac{n}{e} \right)^{n}$ as  $n \to \infty$, applying this approximation to the sum leads to the conclusion that there exists a sufficiently small  $\alpha_{*}>0$, such that
\begin{align*}
    \sup_{u \in W^{s,p}_{0}(\Omega \setminus K), \ [u]_{W^{s,p}(\Omega)} \leq 1}  \int_{\Omega} \Phi_{d,s} \left( \alpha  |u(x)|^{\frac{d}{d-s}}  \right) \, \frac{dx}{\delta^{d}_{K}(x)}     < \infty , \quad   \forall \ \alpha \in [0, \alpha_{*}).
\end{align*}
This establishes the existence of  $\alpha_{*}$ in Theorem  \ref{Th : Trud-Moser inequ}.

\bigskip

Now, we establish that for any $\alpha > 0$, the function $\Phi_{d,s} \!\left( \alpha \, |u(x)|^{\frac{d}{d-s}} \right) \frac{1}{\delta^{d}_{K}(x)}  \in L^{1}(\Omega)$ for any $u \in W^{s,p}_{0} (\Omega \setminus K)$. Following the ideas in \cite{Adi2025, Iula2017, ruf2019}, we obtain the following corollary.

\begin{cor}\label{corollary 1}
Let $\Omega$ be a bounded Lipschitz domain in $\mathbb{R}^{d}$, and let $K \subset \Omega$ be a smooth compact set of class $C^{0,1}$ of codimension $k$, where $k \in \mathbb{N}$ and $1 < k < d$.  
Let $p > 1$ and $s \in (0,1)$ be such that $sp = d$.  
Then for any $u \in W^{s,p}_{0} (\Omega \setminus K)$ and $\alpha > 0$, we have
\begin{equation}
   \int_{\Omega}  \Phi_{d,s} \!\left( \alpha \, |u(x)|^{\frac{d}{d-s}} \right) \frac{1}{\delta^{d}_{K}(x)} \, dx < \infty,
\end{equation}
where $\Phi_{d,s}$ is defined in \eqref{Defn: Phi}.
\end{cor}
\begin{proof}
   Let $u \in W^{s,p}_{0}(\Omega \setminus K)$ and fix $\alpha > 0$.  
Since $C^{\infty}_{c}(\Omega \setminus K)$ is dense in $W^{s,p}_{0}(\Omega \setminus K)$, we may choose $f \in C^{\infty}_{c}(\Omega \setminus K)$ and $g \in W^{s,p}_{0}(\Omega \setminus K)$ such that
\begin{equation*}
  u = f + g
\quad\text{and}\quad
[g]_{W^{s,p}(\Omega)} \leq \frac{1}{2} \left( \frac{\alpha_{*}}{2\alpha} \right)^{\!\frac{d-s}{d}},  
\end{equation*}
where $\alpha_{*}$ is the constant appearing in Theorem \ref{Th : Trud-Moser inequ}. By the triangle inequality, for $n \geq j_{p}-1$ we have
\begin{equation*}
   |u(x)|^{\frac{nd}{d-s}}
= |f(x) + g(x)|^{\frac{nd}{d-s}}
\leq 2^{\frac{nd}{d-s}-1} |f(x)|^{\frac{nd}{d-s}}
+ 2^{\frac{nd}{d-s}-1} |g(x)|^{\frac{nd}{d-s}}. 
\end{equation*}
It follows that
\begin{equation*}
    \sum_{n=j_{p}-1}^{\infty}  \frac{\alpha^{n}|u(x)|^{\frac{nd}{d-s}}}{n!} \leq \frac{1}{2}  \sum_{n=j_{p}-1}^{\infty}  \frac{\alpha^{n} 2^{\frac{nd}{d-s}} |f(x)|^{\frac{nd}{d-s}}}{n!} + \frac{1}{2}\sum_{n=j_{p}-1}^{\infty}  \frac{ \alpha^{n} 2^{\frac{nd}{d-s}} |g(x)|^{\frac{nd}{d-s}}}{n!}.
\end{equation*}
Therefore,
\begin{equation*}
 \Phi_{d,s} \! \left( \alpha |u(x)|^{\frac{d}{d-s}} \right)
\leq \frac{1}{2} \Phi_{d,s} \! \left( 2^{\frac{d}{d-s}} \alpha |f(x)|^{\frac{d}{d-s}} \right)
+ \frac{1}{2} \Phi_{d,s} \! \left( 2^{\frac{d}{d-s}} \alpha |g(x)|^{\frac{d}{d-s}} \right),  
\end{equation*}
where $\Phi_{d,s}$ is defined in \eqref{Defn: Phi}. Multiplying both sides by $\delta_{K}(x)^{-d}$ and using 
\begin{equation*}
    [g]_{W^{s,p}(\Omega)} \leq \frac{1}{2} \left( \frac{\alpha_{*}}{2\alpha} \right)^{\!\frac{d-s}{d}},
\end{equation*}
we obtain
\begin{align}\label{eqn113}
\Phi_{d,s} \!\left( \alpha |u(x)|^{\frac{d}{d-s}} \right) \frac{1}{\delta^{d}_{K}(x)}
&\leq \frac{1}{2} \Phi_{d,s} \!\left( 2^{\frac{d}{d-s}} \alpha |f(x)|^{\frac{d}{d-s}} \right) \frac{1}{\delta^{d}_{K}(x)} \nonumber \\
&\quad + \frac{1}{2} \Phi_{d,s} \!\left( \frac{\alpha_{*}}{2}
\left( \frac{|g(x)|}{[g]_{W^{s,p}(\Omega)}} \right)^{\frac{d}{d-s}} \right) \frac{1}{\delta^{d}_{K}(x)}.
\end{align}
Since $f \in C^{\infty}_{c}(\Omega \setminus K)$, its support is compact.  
Thus, there exist $C > 0$ and a compact set $F \subset \Omega \setminus K$ such that 
$|f(x)| \leq C\,\chi_{F}(x)$ for all $x \in \Omega$, where $\chi_{F}$ is the indicator function of $F$.  
Therefore,
\begin{align*}
\Phi_{d,s} \!\left( 2^{\frac{d}{d-s}} \alpha |f(x)|^{\frac{d}{d-s}} \right) \frac{1}{\delta^{d}_{K}(x)}
&= \sum_{n = j_{p}-1}^{\infty} \frac{\alpha^{n}}{n!} ( 2|f(x)| )^{\frac{nd}{d-s}} \frac{1}{\delta^{d}_{K}(x)} \\
&\leq \left( \sum_{n = j_{p}-1}^{\infty} \frac{\alpha^{n}}{n!} (2C)^{\frac{nd}{d-s}} \right)
\frac{\chi_{F}(x)}{\delta^{d}_{K}(x)}
\in L^{1}(\Omega).
\end{align*}
Finally, combining the above estimate with \eqref{eqn113} and applying Theorem \ref{Th : Trud-Moser inequ} to the function 
$\frac{g}{[g]_{W^{s,p}(\Omega)}} \in W^{s,p}_{0}(\Omega \setminus K)$,  
we conclude that
\begin{equation*}
\int_{\Omega} \Phi_{d,s} \!\left( \alpha |u(x)|^{\frac{d}{d-s}} \right) \frac{1}{\delta^{d}_{K}(x)} \, dx < \infty,   
\end{equation*}
which proves the corollary.
\end{proof}

Now we establish a proposition showing that the spaces $\widetilde{W}^{s,p}_{0} (\Omega \setminus K)$ and $W^{s,p}_{0} (\Omega \setminus K)$ are equivalent in the fractional Sobolev norm.

\begin{prop}\label{Proposition 1}
Let $p>1$ and $s \in (0,1)$ satisfy $sp=d$, where $d \geq 2$. Let
$u \in W^{s,p}_{0}(\Omega \setminus K)$, and let
$\widetilde{u} \in \widetilde{W}^{s,p}_{0}(\Omega \setminus K)$ be its zero extension, as defined in \eqref{Extension Function}. Then there exists a constant
$C=C(d,s,k,\Omega,K)>0$ such that
    \begin{align}
     \int_{\Omega} \int_{\Omega} \frac{|u(x)-u(y)|^{p}}{|x-y|^{2d}} \, dx \, dy \leq   \int_{\mathbb{R}^{d}} \int_{\mathbb{R}^{d}} & \frac{|\widetilde{u}(x)-\widetilde{u}(y)|^{p}}{|x-y|^{2d}} \, dx \, dy \nonumber \\ &  \leq C \int_{\Omega} \int_{\Omega} \frac{|u(x)-u(y)|^{p}}{|x-y|^{2d}} \, dx \, dy.
    \end{align}
    In particular, $\widetilde{W}^{s,p}_{0} (\Omega \setminus K) = W^{s,p}_{0} (\Omega \setminus K)$ as Banach spaces.
\end{prop}
\begin{proof}
Let $u \in W^{s,p}_{0}(\Omega \setminus K)$, and let
$\widetilde{u} \in \widetilde{W}^{s,p}_{0}(\Omega \setminus K)$ be its zero extension, as defined in \eqref{Extension Function}. Since
$\Omega \times \Omega \subset \mathbb{R}^{d} \times \mathbb{R}^{d}$, we have
\begin{equation*}
    \int_{\Omega} \int_{\Omega}
    \frac{|u(x)-u(y)|^{p}}{|x-y|^{2d}}
    \, dx \, dy
    \leq
    \int_{\mathbb{R}^{d}} \int_{\mathbb{R}^{d}}
    \frac{|\widetilde{u}(x)-\widetilde{u}(y)|^{p}}{|x-y|^{2d}}
    \, dx \, dy.
\end{equation*}
This proves one direction of the inequality. To prove the reverse inequality, note that $\widetilde{u}\equiv 0$ on $\mathbb{R}^{d}\setminus\Omega$. Therefore,
\begin{align}\label{eqn1 prop}
   \int_{\mathbb{R}^{d}} \int_{\mathbb{R}^{d}} \frac{|\widetilde{u}(x)-\widetilde{u}(y)|^{p}}{|x-y|^{2d}} \, dx \, dy
   &=  \int_{\Omega} \int_{\Omega} \frac{|u(x)-u(y)|^{p}}{|x-y|^{2d}} \, dx \, dy  \nonumber \\
   &\quad + 2 \int_{\Omega} |u(x)|^{p} \int_{\mathbb{R}^{d} \setminus \Omega} \frac{1}{|x-y|^{2d}} \, dy \, dx.
\end{align}
For any $x \in \Omega$, let $\delta_{x} :=  \min \{ \delta_{\Omega} (x) , \delta_{K}(x) \}$, where $\delta_{\Omega}(x)$ denotes the distance from $x$ to $\mathbb{R}^d \setminus \Omega$, and $\delta_{K}(x)$ the distance from $x$ to $K$. Then 
\begin{equation*}
    \mathbb{R}^{d} \setminus \Omega \subset \mathbb{R}^{d} \setminus B_{\delta_{x}}(x), \quad \forall \ x \in \Omega,
\end{equation*}
where $B_{\delta_{x}}(x)$ is the ball of radius $\delta_x > 0$ centered at $x$ in $\mathbb{R}^{d}$. Therefore,
\begin{align*}
    \int_{\Omega} |u(x)|^{p}  \int_{\mathbb{R}^{d} \setminus \Omega} \frac{1}{|x-y|^{2d}} \, dy \, dx & \leq C \int_{\Omega} |u(x)|^{p} \left( \int_{\delta_{x}}^{\infty} r^{-1-d} \, dr \right) \, dx \\ & \leq C \int_{\Omega} |u(x)|^{p} \left(  \int_{\delta_{\Omega}(x)}^{\infty} r^{-1-d} \, dr \right) \, dx \\ & \hspace{4mm} + C \int_{\Omega} |u(x)|^{p} \left(  \int_{\delta_{K}(x)}^{\infty} r^{-1-d} \, dr \right) \, dx  \\ & \leq C \left( \int_{\Omega} \frac{|u(x)|^{p}}{\delta^{d}_{\Omega}(x)} \, dx + \int_{\Omega} \frac{|u(x)|^{p}}{\delta^{d}_{K}(x)} \, dx \right).
\end{align*}
By applying \cite[Theorem $1.1$]{dyda2004} to the first integral on the right-hand side, and Theorem \ref{Th: Fractional Hardy sp=k} from this article to the second term on the right-hand side, we obtain
\begin{equation}\label{eqn2 prop}
    \int_{\Omega} |u(x)|^{p}  \int_{\mathbb{R}^{d} \setminus \Omega} \frac{1}{|x-y|^{2d}} \, dy \, dx
    \leq
    C \int_{\Omega} \int_{\Omega} \frac{|u(x)-u(y)|^{p}}{|x-y|^{2d}} \, dx \, dy,
\end{equation}
where $C = C(d,s,k,\Omega,K) > 0$.  
Combining the inequalities \eqref{eqn1 prop} and \eqref{eqn2 prop} completes the proof.
\end{proof}


\subsection{Existence of \texorpdfstring{$\alpha^{*}_{s,d}$}{alpha star} in Theorem  \ref{Th : Trud-Moser inequ}}\label{subsec2}  To prove the sharpness of $\alpha^{*}_{s,d}$ in Theorem \ref{Th : Trud-Moser inequ}, we consider the following sequence of Moser-type functions. For $\epsilon > 0$, define $u_{\epsilon} : \mathbb{R}^{d} \to \mathbb{R}$ by 
\begin{equation}\label{moser-type fn}
   u_{\epsilon} (x) = \begin{dcases}
        |\ln \epsilon|^{\frac{d-s}{d}} , & |x| \leq \epsilon \\ 
       \frac{|\ln|x||}{|\ln \epsilon|^{\frac{s}{d}}}, & \epsilon < |x|<1  \\
       0, & |x| \geq 1.
    \end{dcases}
\end{equation}
A result by Parini and Ruf \cite[Proposition 5.1]{ruf2019} shows that for $d \geq 2$, the fractional seminorm of $u_{\epsilon}$ satisfies
\begin{equation}\label{ruf-limit}
\lim_{\epsilon \to 0} [u_{\epsilon}]^{p}_{W^{s,p}(\mathbb{R}^{d})} =  \frac{2 (d \omega_{d})^{2} \Gamma (p+1) }{d!}  \sum_{n=0}^{\infty} \frac{(d+n-1)!}{n!} \frac{1}{(d+2n)^{p}} ,
\end{equation}
where $\omega_{d}$ denotes the volume of the unit ball in $\mathbb{R}^{d}$.
Let us assume that $0 \in \Omega \setminus K$. Then there exists $r > 0$ such that $B_{r}(0) \subset \Omega \setminus K$. With proper scaling, we can show that $u_{\epsilon} \in \widetilde{W}^{s,p}_{0}(\Omega \setminus K)$. Therefore, we may further assume that $B_{1}(0) \subset \Omega \setminus K$. Let $\alpha > \alpha^{*}_{s,d}$. Using the definition of $\alpha^{*}_{s,d}$ in \eqref{Defn alpha star d}, the assumption
$\alpha>\alpha^{*}_{s,d}$, and \eqref{ruf-limit}, we obtain
\begin{equation*}
\lim_{\epsilon\to0}
\alpha [u_{\epsilon}]^{-\frac{d}{d-s}}_{W^{s,p}(\mathbb{R}^{d})}
=
\frac{\alpha}{\alpha^{*}_{s,d}}\,d
>d.
\end{equation*}
Hence, we may fix a constant $\beta$ such that $d<\beta<\frac{\alpha}{\alpha^{*}_{s,d}}\,d$. Moreover, for every $x\in B_{\epsilon}(0)$,
\[
u_{\epsilon}(x)=|\ln\epsilon|^{\frac{d-s}{d}},
\]
and hence
\begin{equation*}
|u_{\epsilon}(x)|^{\frac{d}{d-s}}
=
|\ln\epsilon|
\longrightarrow\infty
\qquad\text{as }\epsilon\to0.
\end{equation*}
Therefore, using the fact that $\Phi_{d,s}(t)\geq \frac12\exp(t)$ for all sufficiently large $t>0$, we may choose $\epsilon>0$ sufficiently small so that
\begin{equation}\label{beta-choice}
\alpha [u_{\epsilon}]^{-\frac{d}{d-s}}_{W^{s,p}(\mathbb{R}^{d})}
\geq
\beta>d,
\end{equation}
and
\begin{equation}\label{phi-exp}
\Phi_{d,s}\left(\alpha |u_{\epsilon}(x)|^{\frac{d}{d-s}}\right)
\geq
\frac12
\exp\left(\alpha |u_{\epsilon}(x)|^{\frac{d}{d-s}}\right),
\qquad
\forall\,x\in B_{\epsilon}(0).
\end{equation}
Recall that $\delta_{K}(x)$ denotes the distance from $x$ to the set $K$.
Since $\delta_{K}^{-d}(x) \geq C_{0} > 0$ for all $x \in \Omega$, we obtain
\begin{align*}
\int_{\Omega} \Phi_{d,s} \left( \alpha \left( \frac{|u_{\epsilon}(x)|}{[u_{\epsilon}]_{W^{s,p}(\mathbb{R}^{d})}} \right)^{\frac{d}{d-s}} \right) \frac{dx}{\delta^{d}_{K}(x)} 
 \geq C_{0} \int_{\Omega} \Phi_{d,s} \left( \alpha \left( \frac{|u_{\epsilon}(x)|}{[u_{\epsilon}]_{W^{s,p}(\mathbb{R}^{d})}} \right)^{\frac{d}{d-s}} \right) dx.
\end{align*}
Applying \eqref{phi-exp}, we further have
\begin{align*}
    \int_{\Omega} \Phi_{d,s} \left( \alpha \left( \frac{|u_{\epsilon}(x)|}{[u_{\epsilon}]_{W^{s,p}(\mathbb{R}^{d})}} \right)^{\frac{d}{d-s}}  \right) \frac{dx}{\delta^{d}_{K}(x)} & \geq C_{0} \int_{\Omega} \Phi_{d,s} \left( \alpha \left( \frac{|u_{\epsilon}(x)|}{[u_{\epsilon}]_{W^{s,p}(\mathbb{R}^{d})}} \right)^{\frac{d}{d-s}}  \right) dx  \\ &  \geq  \frac{C_{0}}{2} \int_{B_{\epsilon}(0)} \exp \left( \left( \alpha \left( \frac{|u_{\epsilon}(x)|}{[u_{\epsilon}]_{W^{s,p}(\mathbb{R}^{d})}} \right)^{\frac{d}{d-s}}  \right)  \right)\, dx .
\end{align*}
Using the estimate \eqref{beta-choice}, we obtain
\begin{align*}
    \int_{\Omega} \Phi_{d,s} \left( \alpha \left( \frac{|u_{\epsilon}(x)|}{[u_{\epsilon}]_{W^{s,p}(\mathbb{R}^{d})}} \right)^{\frac{d}{d-s}}  \right) \frac{dx}{\delta^{d}_{K}(x)} & \geq \frac{C_{0}}{2} \exp \left( - \beta \ln \epsilon \right) \int_{B_{\epsilon}(0)} dx \\ & = C \epsilon^{d- \beta}.
\end{align*}
Since $\beta > d$, the factor $\epsilon^{d-\beta} \to \infty$ as $\epsilon \to 0$.
Therefore, 
\begin{align}
   \sup_{u \in \widetilde{W}^{s,p}_{0}(\Omega \setminus K), \ [u]_{W^{s,p}(\mathbb{R}^{d})} \leq 1}  \int_{\Omega} \Phi_{d,s} \left( \alpha  |u(x)|^{\frac{d}{d-s}}  \right) \, \frac{dx}{\delta^{d}_{K}(x)}  &   = \infty , \quad   \forall \ \alpha \in (\alpha^{*}_{s,d}, \infty).
\end{align}
This completes the proof of Theorem \ref{Th : Trud-Moser inequ}.

\bigskip

\textbf{Acknowledgement:} The author gratefully acknowledges the financial support of the Anusandhan National Research Foundation (ANRF) through the National Post Doctoral Fellowship (PDF/2025/004611). The author also thanks the Theoretical Statistics and Mathematics Unit, Indian Statistical Institute, Delhi Centre, India, for providing a supportive and stimulating research environment.


\end{document}